\documentclass[  headinclude,  open=right, toc=listof, toc=bibliography]{amsart}
\usepackage[T1]{fontenc} 
\usepackage[utf8]{inputenc}
\usepackage{amsmath}
\usepackage{amssymb}
\usepackage{amsthm}
\usepackage{dsfont}
\usepackage{tikz-cd} 
\usepackage{mathtools}
\usepackage{enumerate}
\usepackage{mathrsfs} \usepackage{color} 	
\usepackage[backend=biber,url=false,isbn=false, doi=false, bibstyle=numeric, giveninits=true,  maxbibnames=9]{biblatex}
\usepackage{bbm}
\usepackage{xcolor}
\usepackage{graphicx}
\usepackage[normalem]{ulem}

\colorlet{colorCE}{red}
\colorlet{colorAL}{blue!50}

\usepackage[colorinlistoftodos,prependcaption,textsize=small]{todonotes} 
\colorlet{colorAL}{blue!50}
\colorlet{colorCE}{red!50}

\usepackage{geometry}

\newtheorem{theorem}{Theorem}[section]
\newtheorem{definition}[theorem]{Definition}

\newtheorem{corollary}[theorem]{Corollary}
\newtheorem{proposition}[theorem]{Proposition}
\newtheorem{lemma}[theorem]{Lemma}
\theoremstyle{definition}

\newcommand{\R}{\mathbb{R}}
\newcommand{\C}{\mathbb{C}}
\newcommand{\N}{\mathbb{N}}

\newcommand{\T}{\mathbb{T}}

\AtEveryBibitem{\clearfield{month}}
\AtEveryBibitem{\clearfield{pages}}

\begin{document}

\title[A generalization  of Krein`s extension formalism]{A generalization  of Krein`s extension formalism for symmetric relations with deficiency index $(1,1)$ }
\author{Christian Emmel}
\begin{abstract}
Let $S$ be a  symmetric relation  with deficiency index $(1,1)$. In this article, we extend  Krein`s  resolvent formalism in order to describe all, not necessarily self-adjoint, extensions $S \subset \tilde{A}$ with $\varrho(\tilde{A})\neq \emptyset$. The corresponding $Q$-functions  turn out to be quasi-Herglotz functions. We will use their structure to characterize the spectrum of such extensions.
Finally, we also provide a model for such an extension on a reproducing kernel Hilbert space  when $S$ is simple. 
\end{abstract}
\maketitle

\section{Introduction}
The extension theory of symmetric relations  with deficiency index $(1,1)$ has been a continuously evolving field of research for over  a century, see the surveys \cite{SurveyWeylFunction}, \cite{Beherndt} and the references therein. Historically,  the main focus has been on self-adjoint extensions, whose properties have been extensively studied through the use of  Weyl-functions (or $Q$-functions). In this article, we extend this approach  to characterize all, not-necessarily self-adjoint, extensions which have a nonempty resolvent set. Such extensions have attracted substantial interested on various occasions recently, see in particular \cite{SIngPertubation1}, \cite{SingPertubation2}, \cite{RegPert2},  \cite{RegularPert} and \cite{BehrendtNon-self} as well as the references therein. 
\par\smallskip
First, let us describe the two well-established main components of the self-adjoint extension theory. To this end, let $S$ be a symmetric relation  with deficiency index $(1,1)$ and fix some self-adjoint extensions $A$.
At the heart of the self-adjoint extension theory lies \emph{one} $\varphi$-field given by the following formula:
\begin{align*}
\varphi_\phi(w) := \left(I+(w-i)(A-w)^{-1} \right) \phi   \quad   \forall   w \in \varrho(A).
\end{align*}
Here, the element $\phi $ is a non-trivial element in the defect space $\eta_i(S)$ at the point $i$. Then every other self-adjoint extension $\tilde{A}$ is given by Krein`s resolvent formula:
\begin{align}\label{forBasic1}
    (\tilde{A}-\zeta)^{-1}=(A-\zeta)^{-1}-\frac{[\ \cdot \ , \varphi_\phi(\overline{\zeta}) ]_\mathcal{H}}{q(\zeta)+c} \cdot \varphi_\phi(\zeta).
\end{align}
Here, the \emph{parameter} is the arbitrary real constant $c\in \R$. Moreover,  $q$ is the so called  $Q$-function \cite{KrienLanger2}, which is (essentially) determined by the following relation: 
\begin{align*}
    \frac{q(\zeta)-q(\overline{w})}{\zeta-\overline{w}}= [\varphi_\phi(\zeta), \varphi_\phi(w)]_\mathcal{H}
    , \quad \zeta,w \in \varrho(A).
\end{align*}
It is known that the $Q$-function $q \vcentcolon \C \setminus \R \rightarrow \C$ is a  Herglotz-Nevanlinna function, which means that it maps the upper-half plane  $\C^+$ into $\overline{\C^+}$ and satisfies the symmetry condition $q(\overline{\zeta})= \overline{q(\zeta)}$. Such a function possesses an integral representation, which allows us to express $q$ as follows: 
\begin{align*}
    q(\zeta)= a+b \zeta + \int_{\R} \frac{1+t \zeta}{t-\zeta} \cdot \mathrm{d}\nu
\end{align*}
Here, $a\in \R$, $b \geq 0$ are constants and $\nu$ is a finite \emph{non-negative} Borel measure. 
\par\smallskip
Finally, the function $q+c$ encodes the spectral properties of the self-adjoint extension $\tilde{A}$. For example, its zeros of order $1$ correspond to eigenvalues of algebraic multiplicity $1$. It is therefore a very useful tool to investigate the spectral properties of $\tilde{A}$. 
\par\smallskip
Now let us  consider  \emph{regular} extensions of $S$, by  which we mean relations $\tilde{A}$ satisfying $S \subset \tilde{A}$ and $\varrho(\tilde{A})\neq \emptyset$. In this article, we show that the class of regular extension is again given  by a Krein type resolvent formalism, which is now of the following form: 
\begin{align*}
    (\tilde{A}-\zeta)^{-1}=(A-\zeta)^{-1}-\frac{[\ \cdot \ , \varphi_\phi(\overline{\zeta}) ]_\mathcal{H}}{q(\zeta)+c} \cdot \varphi_v(\zeta).
\end{align*}
 Here, $\phi\in \eta_i(S)\setminus \{0\}$ is a \emph{fixed} defect element and $A$ is a \emph{fixed} self-adjoint extension, while  the \emph{parameters} are arbitrary elements  $v \in \mathcal{H}$ and  $c \in \C$. Finally, thefunction $q$ is (essentially) determined by the following identity: 
\begin{align*}
    \frac{q(\zeta)-q(\overline{w})}{\zeta-\overline{w}}= [\varphi_v(\zeta), \varphi_\phi(w)]_\mathcal{H}
    , \quad \zeta,w \in \varrho(A).
\end{align*}
The function $q$ 
is in general not a Herglotz-Nevanlinna function any more, but it is a quasi-Herglotz function \cite{QuasiHN}.   These are functions that admit an integral representation of the following form: 
\begin{align*}
    q(\zeta)= a+b \zeta + \int_{\R} \frac{1+t \zeta}{t-\zeta} \cdot \mathrm{d}\nu
\end{align*}
Here, $a,b\in \C$ and $\nu$ is a finite \emph{complex} Borel measure. 
It is important to note that $q$ also depends on $v$.
In summary, we establish that regular extensions can be treated in a slightly more involved but similar way as self-adjoint extensions.
\par\smallskip
As in the self-adjoint case, the function $q+c$ encodes the spectral properties of $\tilde{A}$. We will use these functions to characterize the spectrum of such extensions. In particular, we will determine all possible eigenvalue distributions on $\C^+$ (or $\C^-)$.  
\par\smallskip
Finally, let us consider two examples. First, let $A$ be a  self-adjoint operator and $x\in \mathcal{H}\setminus \{0\}$ be an arbitrary non-zero element. Then the restriction 
\begin{align}\label{For:Basic4390}
    S= A_{|\{u\in \mathrm{dom}(A) \vcentcolon \ [u,x]=0\}}
\end{align}
is a symmetric operator with deficiency index $(1,1)$. If $A$ is bounded, then the self-adjoint extensions $S$ are  (essentially) just the  perturbations 
\begin{align*}
    A+c \cdot [\cdot,x] \cdot x,
\end{align*}
where $c\in \R$ is a real constant. On the other hand,  the regular extensions are in this case simply the asymmetric perturbations 
\begin{align*}
     A+[\cdot,x]  \cdot y,
\end{align*}
where $y \in \mathcal{H}$ is an arbitrary element.  Asymmetric perturbations of  self-adjoint operators with pure point spectrum have been investigated in a series of papers recently \cite{RegPert2}, \cite{RegularPert}.
\par\smallskip
The situation is more complicated when $A$ is not bounded. Then, in order  to obtain a complete picture, one would need to consider perturbations by singular elements $x,y \in \mathcal{H}_{-2}$ as well. We will not go into details here. However,  we take the opportunity to announce that in a follow-up paper we will give an interpretation of regular extensions as asymmetric singular perturbations.
\par\smallskip
Next, we  discuss simple symmetric operators with deficiency index $(1,1)$. These are characterized by the following minimality condition: 
\begin{align*}
         \overline{\mathrm{span}}\{\eta_{\lambda}(S) \vcentcolon \lambda \in \C \setminus \R \}= \mathcal{H}.
\end{align*}
Here, $\eta_{\lambda}(S)$ denotes the defect space at $\lambda$. Equivalently, one could also demand that $S$ has a self-adjoint extension which has a cyclic vector. 
In any case, this condition is satisfied by many interesting differential operators, for example by Sturm-Lioville differential expressions on the interval $[0,\infty)$ under weak and reasonable assumptions. 
\par\smallskip
Over the years, various models for their extensions have been developed using methods from complex analysis and integral operators \cite{Treil}, \cite{SIngPertubation1}. 
Interestingly, simple symmetric operators and their (self-adjoint) extensions are also closely connected to certain well-studied function spaces \cite{DeBranges}. In this paper, we will use this connection  to give a complete and surprisingly simple function-theoretic description of their (regular) extension theory. 
\par\smallskip
To this end, let 
$h$ be a Herglotz-Nevanlinna function  and $\mathcal{L}(h)$ the reproducing kernel Hilbert space  associated to the  Nevanlinna kernel 
\begin{align*}
    N_h(z,w)=\frac{h(\zeta)-h(\overline{w})}{\zeta-\overline{w}}.
\end{align*}
Then the operator 
\begin{align*}
    S \vcentcolon \{f \in \mathcal{L}(h) \vcentcolon z \cdot f \in \mathcal{L}(h) \} \subset \mathcal{L}(h) \rightarrow \mathcal{L}(h), \quad  S(f)(z)=z \cdot f(z).
\end{align*}
is a model for simple symmetric operators with deficiency index $(1,1)$. 
Moreover,  $S$ has a self-adjoint extension $A$ such that $(A-w)^{-1}$ is the difference quotient operator:
\begin{align*}
(A-w)^{-1}(f)(\zeta)=D_w(f)(\zeta) \vcentcolon =\frac{f(\zeta)-f(w)}{\zeta-w}, \quad f \in \mathcal{L}(h).
\end{align*}
In this particular case, the regular extensions can be described in simple function-theoretic terms. More precisely, let us consider the set of analytic functions
\begin{align*}
     \mathcal{M}(h) \vcentcolon &= \{ g \vcentcolon \C\setminus \R \rightarrow \C \text{ analytic } \ | \  D_w(g) \in \mathcal{L}(h) \quad  \forall w \in \C \setminus \R \}.
\end{align*}
We will show that the space $\mathcal{M}(h)$ consists of quasi-Herglotz Nevanlinna functions and has a simple representation  involving  the (non-negative) measure $\nu$ representing $h$. In any case, any such a function $g \in \mathcal{M}(h)$ defines a regular extension of $S$ via the formula
\begin{align}\label{IntroFor2}
(A_g-w)^{-1}(f)=D_w(f)-\frac{f(w)}{g(w)} \cdot D_w(g), \quad f \in \mathcal{L}(h).
\end{align}
And conversely, every regular extension is of that form. 
Now let us additionally  assume that  $g$ is not identically zero on neither $\C^+$ nor $\C^-$. Then  $A_g$ fits into the following commuting diagram:  
\[\begin{tikzcd}
\mathcal{L}(h) \arrow{r}{(A_g-w)^{-1}} \arrow[swap]{d}{\frac{1}{g}} & \mathcal{L}(h)  \\
\frac{1}{g} \mathcal{L}(h) \arrow{r}{D_w} & \frac{1}{g} \mathcal{L}(h) \arrow{u}{g}
\end{tikzcd}
\]
Here, the multiplication by $\frac{1}{g}$ and $g$ denote the conjugation of a reproducing kernel space. We see that under mild assumptions, a regular extension of a simple symmetric operators acts as the difference quotient operator on a reproducing kernel Hilbert space.
It turns out that many interesting reproducing kernel Hilbert spaces, for example model spaces, can be realized this way. 
\par\smallskip
In summary, we have developed a model for such regular extensions, which is different in nature of the one  developed in 
\cite{SIngPertubation1}. Here, we emphasize that  our approach works without any of the simplifying assumptions on $S$ (that $S$ is simple and that a self-adjoint extension has  only singular spectrum) that were made in \cite{SIngPertubation1}. 
This paper  includes the only complete description of non-self adjoint extensions of a certain class of symmetric operators with deficiency-index $(1,1)$ that we are aware of. 
\par\smallskip
Finally, another interesting approach is looking at proper extensions of dual pairs, which was pioneered by Malamud and Mogilevskii in 2002 \cite{DualPairs}. Given a symmetric relation $S$ with deficiency index $(1,1)$ and some relation $B$ satisfying $S \subset B^*$,  they described all proper extensions $S \subset \tilde{A} \subset B^*$ via a Krein-type formular. 
\section{Preliminaries}
\subsection{Linear relations}
In accordance with the classical self-adjoint extension theory, we will consider linear relations instead of merely linear operators.
Here, a linear relation $A$ on a Hilbert space $\mathcal{H}$  is a linear subspace $A \subset \mathcal{H}\times\mathcal{H}$. 
A linear operator on $\mathcal{H}$ can be identified  with its graph, which is then a linear relation. The spectrum and the resolvent set as well as  the adjoint are defined similarly as for operators,  see e.g. \cite{FuncCalc} for details. In this article, we will mostly work with linear relations in terms of their resolvents, which can be characterized as follows
\cite[Theorem A.13]{FuncCalc}:
\begin{proposition}\label{Prop:PseudoRes0}
    Let $A \subset \mathcal{H}\times\mathcal{H}$ be a closed linear relation such that $\varrho(A)\neq \emptyset.$ Then its resolvent satisfies the resolvent identity: 
    \begin{align*}
        (w-z) \cdot (A-w)^{-1}(A-z)^{-1}=(A-w)^{-1}-(A-z)^{-1}\quad \forall w,z \in \varrho(A).
    \end{align*}
    Conversely, let $\Omega \subset \C$ be a (non-empty) subset of the complex plane and 
    \begin{align*}
        T \vcentcolon \Omega  \rightarrow \mathscr{L}(\mathcal{H})
    \end{align*}
    be a map  
    which takes values in the space  of bounded linear operators on $\mathcal{H}$. Moreover, assume that $T$ 
    satisfies the following identity: 
    \begin{align*}
        (w-z) \cdot T(w)T(z)=T(w)-T(z) \quad \forall w,z \in \Omega.
    \end{align*}
    Then there exists a linear relation $A$ such that $\Omega \subset \varrho(A)$ and $(A-w)^{-1}=T(w)$.
\end{proposition}

\subsection{Symmetric relations}
In this section, we discuss
the definition and basic properties of symmetric relations. A comprehensive introduction to their their theory is given in e.g. \cite{Beherndt}, where all of the upcoming results can be found. 
\begin{definition}
    A closed relation $S$  on a Hilbert space $\mathcal{H}$ is called symmetric if $S \subset S^*$. 
\end{definition}
Right now, we are  interested in self-adjoint extensions of  symmetric relations. In order for them to exist, $S$ must have equal deficiency indices. They  are defined as follows:
\begin{definition}
Let $S$ be a symmetric relation and $w\in \C\setminus \R$. Then the defect subspace $\eta_w(S)$ is defined as 
\begin{align*}
\eta_w(S)\vcentcolon &=  \mathrm{ran}(S-\overline{w})^\perp= \mathrm{ker}(S^*-w).
\end{align*}
The dimension $\mathrm{dim}(\eta_w(S))$ is constant on $\C^+$ and $\C^-$. This allows us to define the deficiency index of $S$ as $(n,m)=(\mathrm{dim}(\eta_i(S)),\mathrm{dim}(\eta_{-i}(S)))$.
\end{definition}
It is well known that symmetric relations have self-adjoint extensions if and only if they have equal deficiency indices: 
\begin{theorem}
    Let $S$ be a symmetric relation on a Hilbert space $\mathcal{H}$. Then $S$ has a self-adjoint extension if and only if $\mathrm{dim}(\eta_{-i}(S))=\mathrm{dim}(\eta_{i}(S))$. 
\end{theorem}
Furthermore, we will need the notion of points of regular type: 
\begin{definition}
    Let $S$ be a symmetric relation. A point $w\in \C$ is called regular, if $(S-w)$ is injective and its inverse is a (not necessarily everywhere defined) bounded operator. We denote the set of regular points by $\gamma(S)$.
\end{definition}
It holds that $\C\setminus \R \subset \gamma(S)$ for every symmetric relation $S$.
Finally, we are particularly interested in simple  symmetric operators, which are defined as follows: 
\begin{definition}
     A symmetric operator $S$ is simple, if the following holds: 
     \begin{align*}
         \overline{\mathrm{span}}\{\eta_{\lambda}(S) \vcentcolon \lambda \in \C \setminus \R \}= \mathcal{H}.
     \end{align*}
\end{definition}

\subsection{Quasi-Herglotz functions}\label{SubSec:QH}
We will establish that regular extensions are closely connected to quasi-Herglotz functions. These functions  naturally extend the class of Herglotz-Nevanlinna functions, and were initially  introduced and studied in \cite{QuasiHN}:  
\begin{definition}\label{Prop:Herglotz}
Let $q \vcentcolon \C\setminus \R \rightarrow \C$ be an analytic function. Then $q$ is a 
\begin{itemize}
    \item Herglotz-Nevanlinna function if   $q(\overline{\zeta})=\overline{q(\zeta)}$ and ${\rm Im } \,q(\zeta)\geq 0$ for ${\rm Im }\,\zeta> 0$,
    \item 
    quasi-Herglotz function if 
\begin{align}\label{QuasiHn}
    q=(h_1-h_2)+i\cdot (h_3-h_4),
\end{align}
where $h_j, \  j \in \{1,2,3,4\}$ are Herglotz-Nevanlinna functions. 
\end{itemize}
\end{definition}
Quasi-Herglotz functions, as Herglotz-Nevanlinna functions,  can be characterized in terms of an integral representation: 
\begin{proposition}\label{Prop:IntRep}
An analytic function $q \vcentcolon \C\setminus \R \rightarrow \C$ is a quasi-Herglotz function if and only if there exist constants $a,b\in \C$ and a finite complex Borel measure  $\nu$ such that
\begin{align*}
    q(\zeta)= a+b \zeta + \int_{\R} \frac{1+t \zeta}{t-\zeta} \cdot \mathrm{d}\nu
\end{align*}
If that is the case, then the data $(a,b, \nu)$ is uniquely determined. 
\end{proposition}
It is well known \cite[Theorem 3.20]{Teschl2009MathematicalMechanics}, that a quasi-Herglotz function  $q$ has an analytic continuation through an interval $(a,b)\subset \R$ if and only if 
\begin{align*}
    (a,b) \subset \R \setminus \sigma(|\nu|).
\end{align*}
Here, $|\nu|$ denotes the total variation of $\nu$ and  $\sigma(|\nu|)$  the topological support, which is given by the following formula:  
\begin{align*}
 \sigma(|\nu|)\vcentcolon = \{x \in \R \vcentcolon |\nu|(O)>0 \text{ for every open neighbourhood $O$ of $x$} \}.
 \end{align*}
We denote the domain of analyticity by $\mathrm{dom}(q) \vcentcolon = \C \setminus \sigma(|\nu|)$.

\subsection{Associated $\varphi$-fields}
In this section, we review some basic properties of  $\varphi$-fields. These  are commonly used in the self-adjoint extension theory. Let us fix a symmetric relation $S$ with deficiency index $(1,1)$ and a self-adjoint extension $S\subset A$. In order to stay consistent with our notation throughout this article, we fix $i\in \varrho(A)$ as the base point of our analysis.
\begin{definition}\label{Def:Gammafield}
Let $v \in \mathcal{H}$ be arbitrary. Then we define the $\varphi$-field $\varphi_v$ associated to $v$ as follows: 
\begin{align*}
\varphi_v(w) := \left(I+(w-i)(A-w)^{-1} \right) v  \quad   \forall   w \in \varrho(A).
\end{align*}
Note that $\varphi_v(i)=v$. Moreover, if $v \in \eta_{i}(S)$ is a defect element, then it holds that $\varphi_v(w) \in \eta_{w}(S)$ for every $w\in \varrho(A)$.
\end{definition}
A standard calculation involving the resolvent identity yields the following identity: 
\begin{align}\label{Basic:Formula2}
    \varphi_v(w) := \left(I+(w-z)(A-w)^{-1} \right) \varphi_v(z)  \quad   \forall  z, w \in \varrho(A).
\end{align}
In other words, the notion of a $\varphi$-field does not depend on the choice of the base point. We record the following identity for later reference: 
\begin{align}\label{Basic:formula}
    \frac{\varphi_v(w)-\varphi_v(z)}{w-z}=(A-w)^{-1}  \varphi_v(z), \quad   \forall  z, w \in \varrho(A). 
\end{align}

\section{Extensions of symmetric relations}
Our standing assumption in this section is that $S$ is a symmetric relation with deficiency index $(1,1)$. Moreover, let us fix a self-adjoint extension $A$ and a defect element $\phi \in \eta_i(S)\setminus\{0\}$. Our theory does not depend on these particular choices, they  should rather be interpreted as a choice of a reference framework.  In any case, our goal is to characterize all extensions with nonempty resolvent set, which we will  call regular henceforth: 
\begin{definition}
    A  closed linear relation $\tilde{A}$ on $\mathcal{H}$ is a regular extension of $S$ if  $\varrho(\tilde{A})\neq \emptyset$ and $S \subset \tilde{A}$.
\end{definition}
\subsection{The $Q$-functions}
Now let us fix some target element $v\in \mathcal{H}$. 
Our first goal is  to define the associated $Q$-functions:
\begin{definition}
    A function $q \vcentcolon \varrho(A) \rightarrow \C$ is called $Q$-function (associated to $v$) if the following holds:
    \begin{align}\label{For:Q-function}
    \frac{q(\zeta)-q(\overline{w})}{\zeta-\overline{w}}= [\varphi_v(\zeta), \varphi_\phi(w)]_\mathcal{H}
    ,  \quad \forall \zeta,w \in \varrho(A).
\end{align}
\end{definition}
This is a straightforward generalization of the idea behind the classical $Q$-function introduced by Krein and Langer in the 1970`s \cite{KrienLanger2}. 
\par\smallskip
It is clear by construction, that $Q$-functions are unique up to an additive complex constant. The existence of such functions is the subject of the next Lemma: 
\begin{lemma}
    A function $q \vcentcolon \varrho(A) \rightarrow \C$ is a  $Q$-function (associated to $v$) if (and only if) the following holds:
    \begin{align}\label{for;4389}
    \frac{q(\zeta)-q(\overline{i})}{\zeta-\overline{i}}= [\varphi_v(\zeta), \varphi_\phi(i)]_\mathcal{H}
    ,  \quad \forall \zeta \in \varrho(A).
\end{align}
 In particular, for every $c \in \C$ the following function is a $Q$-function: 
 \begin{align*}
     q(\zeta) = c+(\zeta-\overline{i}) \cdot [\varphi_v(\zeta), \varphi_\phi(i)]_\mathcal{H}.
 \end{align*}
 And conversely, every $Q$-function is of that form. 
\end{lemma}
\begin{proof}
Let us assume that \eqref{for;4389} holds. Then we calculate: 
\begin{align*}
    [\varphi_v(\zeta), \varphi_\phi(w)]_\mathcal{H}&=
    [\varphi_v(\zeta), \left(I+(w-i)(A-w)^{-1} \right)  \varphi_\phi(i)]_\mathcal{H}\\
    &=[\left(I+(\overline{w}-\overline{i})(A-\overline{w})^{-1} \right)  \varphi_v(\zeta), \varphi_\phi(i)]_\mathcal{H}\\
    &=[ \varphi_v(\zeta), \varphi_\phi(i)]_\mathcal{H}+(\overline{w}-\overline{i}) \cdot [(A-\overline{w})^{-1}   \varphi_v(\zeta), \varphi_\phi(i)]_\mathcal{H}\\
    &=[ \varphi_v(\zeta), \varphi_\phi(i)]_\mathcal{H}+(\overline{w}-\overline{i}) \cdot \left[\frac{\varphi_v(\overline{w})-\varphi_v(\zeta)}{\overline{w}-\zeta}, \varphi_\phi(i)\right]_\mathcal{H}\\
    &=\frac{q(\zeta)-q(\overline{i})}{\zeta-\overline{i}}+
    \frac{(\overline{w}-\overline{i})}{\overline{w}-\zeta} \cdot 
    \left(\frac{q(\overline{w})-q(\overline{i})}{\overline{w}-\overline{i}}
    - \frac{q(\zeta)-q(\overline{i})}{\zeta-\overline{i}}\right)\\
    &=\frac{q(\zeta)-q(\overline{w})}{\zeta-\overline{w}}
\end{align*}
Here, we have used the identities  \eqref{Basic:Formula2} and \eqref{Basic:formula} multiple times. 
\end{proof}
Next, we  prove that $Q$-functions are necessarily Quasi-Herglotz functions:
\begin{theorem}
    There exists a (unique) finite complex Borel measure $\nu$ and a unique constant $b\in \C$ such that the  $Q$-functions associated to $v$ are given by the following formula: 
        \begin{align*}
    q_c(\zeta)
    &=c+b \cdot \zeta +\int_{\R} \frac{1+t \zeta}{t-\zeta} \mathrm{d}\nu(t), \quad c \in \C.
\end{align*}
    We will call the function 
    \begin{align}\label{FOr4}
    q(\zeta)
    &=b \cdot \zeta +\int_{\R} \frac{1+t \zeta}{t-\zeta} \mathrm{d}\nu(t)
\end{align}
normalized $Q$-function henceforth. 
\end{theorem}
The reason for defining a normalized $Q$-function is that we prefer to work with \emph{one} function $q$ \emph{and} a parameter $c \in \C$ instead of a set of functions. This is in accordance with (most of) the literature dealing with the self-adjoint extension theory.
\par\smallskip
A $Q$-function $q$ is uniquely determined by prescribing the constant $c$. However,   
some care should be taken choosing this constant when defining the normalized $Q$-function. Indeed, if $\phi=v$, then we want to end up  with a Herglotz-Nevanlinna function. We will see in the proof, that in that case $\nu\geq 0$ and $b\geq 0$. Consequently, we have to pick a  constant $c \in  [0,\infty)$. In order to stay consistent, we will always choose $c=0$.
\begin{proof}
    First, let $x\in \mathcal{H}$ be an arbitrary element. Then the function 
    \begin{align*}
        h(\zeta) \vcentcolon = -i \cdot [\varphi_x(i), \varphi_x(i)]_\mathcal{H}+(\zeta-\overline{i}) \cdot [\varphi_x(\zeta), \varphi_x(i)]_\mathcal{H}.
    \end{align*}
    is a Herglotz-Nevanlinna function \cite{KrienLanger2}. Moreover, this function satisfies the following identity:  
        \begin{align*}
    \frac{h(\zeta)-h(\overline{i})}{\zeta-\overline{i}}= [\varphi_x(\zeta), \varphi_x(i)]_\mathcal{H} \quad \forall \zeta \in \varrho(A). 
\end{align*}
Next, we polarize the inner product $[\varphi_v(\zeta), \varphi_\phi(i)]_\mathcal{H}$ in the obvious way: 
\begin{align*}
    [\varphi_v(\zeta), \varphi_\phi(i)]_\mathcal{H}=\frac{1}{4} \cdot \bigg(&[\varphi_{v+\phi}(\zeta), \varphi_{v+\phi}(i)]_\mathcal{H}-[\varphi_{v-\phi}(\zeta), \varphi_{v-\phi}(i)]_\mathcal{H}\\
    &+i \cdot [\varphi_{v+i \cdot \phi}(\zeta), \varphi_{v+i \cdot \phi}(i)]_\mathcal{H}-i\cdot [\varphi_{v-i \cdot \phi}(\zeta), \varphi_{v-i \cdot \phi}(i)]_\mathcal{H}\bigg).
\end{align*}
Here, we have used the obvious identities $\varphi_{x+y}=\varphi_{x}+\varphi_{y}$ and $\varphi_{i \cdot x}= i \cdot \varphi_{ x}$ (see Definition \ref{Def:Gammafield}). In any case, we can find four Herglotz-Nevanlinna functions $h_1, h_2, h_3, h_4$ such that 
\begin{align*}
    [\varphi_v(\zeta), \varphi_\phi(i)]_\mathcal{H}
    &=\frac{1}{4} \cdot (D_i(h_1)-D_i(h_2)+i \cdot D_i(h_3)-i \cdot D_i(h_4)) \\
    &= D_i\left(\frac{1}{4} \cdot (h_1-h_2+i \cdot h_3 -i \cdot h_4)\right). 
\end{align*}
Consequently, we see that the Quasi-Herglotz function $q=\frac{1}{4} \cdot (h_1-h_2+i \cdot h_3 -i \cdot h_4)$ is a $Q$-function. This $q$ functions possesses  an integral representation 
\begin{align*}
    q(\zeta)
    &=a+b \cdot \zeta +\int_{\R} \frac{1+t \zeta}{t-\zeta} \mathrm{d}\nu(t). 
\end{align*}
where $a,b$ and $\nu$ are uniquely determined. Every other $Q$-function is given by adding a complex constant. This  completes the proof.  
\end{proof}

\subsection{The main theorem}
We are ready to state our main theorem now. Recall 
our standing assumptions  that $S$ is a symmetric relation with deficiency index $(1,1)$, $A$ is a self-adjoint extension, and $\phi \in \eta_i(S)\setminus\{0\}$ is a defect element. Moreover, we write  $\mathscr{L}(\mathcal{H})$ for  the space of bounded linear operators. 
\begin{theorem}\label{Prop:PseudoRes} Let $v\in \mathcal{H}$ be a target element, $q$  the  associated normalized $Q$-function and $c \in \C$ be such that $q+c$ is not identically zero.  We  define the following operator valued map:
    \begin{align*}
    T_{v,c} \vcentcolon \varrho(A) \setminus \{\zeta \vcentcolon q(\zeta)+c\neq 0\} \rightarrow \mathscr{L}(\mathcal{H}), \quad
    T_{v,c}(\zeta) = (A-\zeta)^{-1}-\frac{[\ \cdot \ , \varphi_\phi(\overline{\zeta}) ]_\mathcal{H}}{q(\zeta)+c} \cdot \varphi_v(\zeta) 
    \end{align*}
    Then there exists a linear relation $\tilde{A}$ such that 
    \begin{align*}
        \varrho(A) \setminus \{\zeta \vcentcolon q(\zeta)+c\neq 0\} \subset \varrho(\tilde{A}), \quad (\tilde{A}-\zeta)^{-1}=T_{v,c}(\zeta). 
    \end{align*}
     Moreover, $\tilde{A}$ is a regular extension of $S$.
     \par\smallskip
     Conversely, let $\tilde{A}$ be a regular  extension of $S$. Then there exists 
     \begin{itemize}
         \item  a target vector $v \in \mathcal{H}$ with associated normalized $Q$-function $q$
         \item a constant $c\in \C$ 
     \end{itemize}
     such that the function $q+c$ is not identically zero, $\varrho(A) \setminus \{\zeta \vcentcolon q(\zeta)+c\neq 0\} \subset \varrho(\tilde{A})$ and 
     \begin{align*}
         (\tilde{A}-\zeta)^{-1}=(A-\zeta)^{-1}-\frac{[\ \cdot \ , \varphi_\phi(\overline{\zeta}) ]_\mathcal{H}}{q(\zeta)+c} \cdot \varphi_v(\zeta).
     \end{align*}
\end{theorem}
Note that a regular extension is determined by the pair of parameters $(v,c)$. We split up the proof into two parts. First, we show that $T_{v,c}$ is indeed the resolvent of a  regular extension. 
\begin{proof}[Proof that $T_{v,c}$ is the resolvent of a  regular extension:]  
We are given a target vector $v\in \mathcal{H}$, the normalized $Q$-function $q$ and $c \in \C$ such that $q+c$ is not identically zero.  For the sake of simplicity, let us set $m=q+c$ and $T=T_{v,c}$. Then the definition of $T$ reads as follows: 
\begin{align*}
    T(\zeta) = (A-\zeta)^{-1}-\frac{[\ \cdot \ , \varphi_\phi(\overline{\zeta}) ]_\mathcal{H}}{m(\zeta)} \cdot \varphi_v(\zeta) 
\end{align*}
In view of Proposition \ref{Prop:PseudoRes0}, we first aim at establishing the following identity: 
\color{black}
\begin{align*}
    (w-z) \cdot T(w)T(z)=T(w)-T(z) \quad \text{ for all $w,z \in \varrho(A): m(w)\neq 0$ and $m(\zeta)\neq 0$ }.
\end{align*}
To this end, let us consider the following scalar valued function:
    \begin{align*}
        k \vcentcolon \varrho(A) \setminus \{\zeta \vcentcolon m(\zeta)\neq 0\} \times \mathcal{H}\rightarrow \C, \quad k(\zeta,f) \vcentcolon = \frac{[f,  \varphi_\phi(\overline{\zeta})]_\mathcal{H}}{m(\zeta)}.
    \end{align*}
    This definition allows us to rewrite the definition of $T(\zeta)$ compactly as follows: 
    \begin{align*}
        T(\zeta)(f) = (A-\zeta)^{-1}(f)-k(\zeta,f) \cdot \varphi_v(\zeta). 
    \end{align*}
    In the first step, we show that the function $k$ satisfies the following resolvent type identity: 
    \begin{align}\label{cPseudoRes}
        (w-z) \cdot k(w,T(z)f)&=k(w,f)-k(z,f)
    \end{align}
    This identity is a consequence of the following calculation: 
    \begin{align*}
           (w-z) \cdot &k(w,T(z)f)
         = (w-z) \cdot \frac{[T(z)f,  \varphi_\phi(\overline{w})]_\mathcal{H}}{m(w)} \\
         &= \frac{1}{m(w)} \cdot 
         (w-z)\cdot [(A-z)^{-1}(f)-k(z,f) \cdot \varphi_v(z),  \varphi_\phi(\overline{w})]_\mathcal{H}\\
         &= \frac{1}{m(w)} \cdot \bigg(  [f, (\overline{w}-\overline{z}) \cdot (A-\overline{z})^{-1} \varphi_\phi(\overline{w})]_\mathcal{H}-k(z,f) \cdot (w-z)\cdot [ \varphi_v(z),  \varphi_\phi(\overline{w})]_\mathcal{H} \bigg)\\
         &=\frac{1}{m(w)} \cdot \bigg( 
         [f,  \varphi_\phi(\overline{w})-\varphi_\phi(\overline{z})]_\mathcal{H}-k(z,f) \cdot (m(w)-m(z)) \bigg)\\
         &=\frac{1}{m(w)} \cdot \bigg( m(w) \cdot k(w,f)-m(z) \cdot k(z,f) -k(z,f) \cdot (m(w)-m(z)) \bigg)\\
         &=k(w,f)-k(z,f)
    \end{align*}
    Here, we have used identity \eqref{Basic:formula} and the fact that $m$ is a $Q$-function and therefore satisfies identity \eqref{For:Q-function} in the fourth equality. Moreover, we have used the identity 
    \begin{align*}
        m(\zeta)\cdot k(\zeta,f)=[f,\varphi_\phi(\overline{\zeta})]_{\mathcal{H}}
    \end{align*}
    twice in the fifth equality.  
    Next, let us consider the following identity:
    \begin{align*}
        (w-z) \cdot T(w)T(z)f &=(w-z) \cdot ((A-w)^{-1}T(z)f-k(w,T(z)f) \cdot \varphi_v(w))\\
        &=(w-z) \cdot \big((A-w)^{-1}T(z)f \big)-(k(w,f)-k(z,f)) \cdot \varphi_v(w)
    \end{align*}
    The term $k(w,T(z)f) \cdot \varphi_v(w)$ was easily dealt with using the just established resolvent-type identity for $k$. We simplify the first summand in the following way: 
    \begin{align*}
        (w-z) \cdot (A-w)^{-1}(T(z)f)&=(w-z) \cdot (A-w)^{-1} \left( (A-z)^{-1}(f)-k(z,f) \cdot \varphi_v(z) \right)\\
        &=(A-w)^{-1}f-(A-z)^{-1}f-k(z,f) \cdot (w-z) \cdot (A-w)^{-1} \varphi_v(z)\\
        &=(A-w)^{-1}f-(A-z)^{-1}f -k(z,f) \cdot(\varphi_v(w)-\varphi_v(z))
    \end{align*}
    Here, we have used the resolvent identity in the first equality and identity \eqref{Basic:formula} in the second equality. Finally, we can add up the two summands and obtain the following identity; 
    \begin{align*}
        (w-z) \cdot &T(w)T(z)f= (w-z) \cdot \big((A-w)^{-1}T(z)f\big)-(k(w,f)-k(z,f)) \cdot \varphi_v(w)\\
        &=(A-w)^{-1}f-(A-z)^{-1}f -k(z,f) \cdot(\varphi_v(w)-\varphi_v(z))-(k(w,f)-k(z,f)) \cdot \varphi_v(w)\\
        &=(A-w)^{-1}f-k(w,f)\cdot \varphi_v(w)-\big((A-z)^{-1}f-k(z,f) \cdot \varphi_v(z)\big)\\
        &=T(w)f-T(z)f.
    \end{align*}
    \color{black}
This means that there exists a linear relation $\tilde{A}$ generating $T$.
\par\smallskip
Finally, let $u\in \mathrm{dom}(S)$. Then $u \in \mathrm{dom}(\tilde{A})$ and  $S \subset \tilde{A}$, which is a consequence of the following calculation: 
\begin{align*}
   (\tilde{A}-w)^{-1} (S-w)u &=   (A-w)^{-1}(S-w) u-\frac{[(S-w)u ,  \varphi_\phi(\overline{w}) ]_\mathcal{H}}{q(w)} \cdot  \varphi_v(w)= u.
\end{align*}
Here, we have used that $\varphi_\phi(\overline{w}) \in \eta_{\overline{w}}(S)=(\mathrm{ran}(S-w))^\perp$.
This completes the proof. 
\end{proof}
Now we treat the converse direction: 
\begin{proof}
Let $\tilde{A}$ be a regular  extension of $S$ and let us 
fix some  $w \in \varrho(\tilde{A}) \setminus \R$. Note that such a point $w$ exists, because $\varrho(\tilde{A})$ is open and non-empty, and that $w \in \varrho(A)$.
Since $S \subset \tilde{A}\cap A$, it follows that
\begin{align*}
    \mathrm{ran}(S-w)\subset  \mathrm{ran}((\tilde{A}\cap A)-w)= \mathrm{ker}((A-w)^{-1}-(\tilde{A}-w)^{-1}),
\end{align*}
where the last equality is a consequence of  \cite[Lemma 1.7.2]{Beherndt}. In summary, we arrive at 
\begin{align*}
     \mathrm{ker}((A-w)^{-1}-(\tilde{A}-w)^{-1})^\perp
     \subset \mathrm{ran}(S-w)^\perp=\eta_{\overline{w}}(S).
\end{align*}
Therefore, the dimension of $\mathrm{ker}((A-w)^{-1}-(\tilde{A}-w)^{-1})^\perp$ is either zero or one, since $\eta_{\overline{w}}(S)$ is one-dimensional. If it is zero, then it holds that
\begin{align*}
    \mathrm{ker}((A-w)^{-1}-(\tilde{A}-w)^{-1})=\mathcal{H},
\end{align*}
which implies that $A=\tilde{A}$. Consequently, we can assume without loss of generality that
\begin{align*}
    \mathrm{ker}((A-w)^{-1}-(\tilde{A}-w)^{-1})^\perp= \eta_{\overline{w}}(S).
\end{align*}
As a result, we can decompose $\mathcal{H}$ as 
\begin{align*}
    \mathcal{H} = \eta_{\overline{w}}(S) \oplus \mathrm{ker}((A-w)^{-1}-(\tilde{A}-w)^{-1}).
\end{align*}
Let us consider the defect element $\varphi_\phi(\overline{w})\in \eta_{\overline{w}}(S)$. Then we can find a vector  $y \in \mathcal{H}\setminus\{0\}$  such that we can express $(\tilde{A}-w)^{-1}-(A-w)^{-1}$ as follows: 
\begin{align}\label{For:BeweisMaximality}
    (\tilde{A}-w)^{-1}-(A-w)^{-1}=-[\ \cdot \ , \varphi_\phi(\overline{w})] \cdot y.
\end{align}
Now define the following target vector:
\begin{align*}
    v\vcentcolon=\left(I+(i-w)(A-i)^{-1} \right)y,
\end{align*}
This means that $\varphi_v(w)=y$. 
Moreover, let  $q$ be the  normalized $Q$-function associated  $v$ and choose $c \in \C$ such that $q(w)+c=1$.
 Then equation \eqref{For:BeweisMaximality} simplifies to 
\begin{align*}
    (\tilde{A}-w)^{-1}-(A-w)^{-1}
    &=-[\ \cdot \ ,\varphi_\phi(\overline{w})]_{\mathcal{H}} \cdot \varphi_v(w)
    = -\frac{[\ \cdot \ , \varphi_\phi(\overline{w})]_{\mathcal{H}}}{q(w)+c} \cdot \varphi_v(w).
\end{align*}
Rearranging terms yields the following identity: 
\begin{align*}
    (\tilde{A}-w)^{-1}
    &=(A-w)^{-1} -\frac{[\ \cdot \ , \varphi_\phi(\overline{w})]_{\mathcal{H}}}{q(w)+c} \cdot \varphi_v(w).
\end{align*}
In summary, we have established our Krein-type resolvent formula at one specific point $w$. However,  we had already established in the converse direction of this proof that there exists a regular extension $\tilde{B}$ such that $\varrho(A) \setminus \{\zeta \vcentcolon q(\zeta)+c\neq 0\} \subset \varrho(\tilde{B})$ and
\begin{align*}
    (\tilde{B}-\zeta)^{-1}= (A-\zeta)^{-1} -\frac{[\ \cdot \ , \varphi_\phi(\overline{\zeta})]_{\mathcal{H}}}{q(\zeta)+c} \cdot \varphi_v(\zeta) \quad \forall \zeta \in \varrho(A) \setminus \{\zeta \vcentcolon q(\zeta)+c\neq 0\}.
\end{align*}
Now since $(\tilde{A}-w)^{-1}=(\tilde{B}-w)^{-1}$ it follows that $\tilde{A}=\tilde{B}$.
 This concludes our proof. 
 \color{black}
\end{proof}
We will completely describe the spectrum of a regular extension  of a simple symmetric operator in Chapter \ref{Sec:hfsdljk}.
\par\smallskip
Finally, let us consider the extension corresponding to some pair $(c,v)$ relative the defect element $\phi$. Moreover, let  $\tilde{\phi}= \lambda \cdot \phi \in \eta_i(S)$ be another defect element, where $\lambda \in \C \setminus \{0\}$. Note that $\eta_i(S)$ is one-dimensional. Then it is straightforward to verify that the pair $\left(c,\frac{1}{\lambda} \cdot v\right)$ correspond to the same extension relative the defect element $\lambda \cdot \phi$. 
\par\smallskip
Changing the base extension $A$ results into a more complicated transformation  that also involves the normalized $Q$-function. We refrain from including the details of the transformation here, and just remark that this does not make a difference conceptually.

\section{Extensions of simple symmetric operators}\label{Ext:SimpleSym}
\subsection{The model}
Every simple symmetric operator with deficiency index $(1,1)$ can be realized as the multiplication operator by the independent variable on a  suitable reproducing kernel Hilbert space. Let us first revisit the definition of reproducing kernels:
\begin{definition}
    Let $U \subset \C$ be an open set and $\mathcal H$ a Hilbert space consisting of analytic functions defined on $U$. Then $\mathcal H$ is a reproducing kernel Hilbert space if there exists a function $ k \vcentcolon U \times U \rightarrow \C$
such that the following holds:
\begin{enumerate}
    \item for every $w \in U$ the function $\zeta \mapsto  k(\zeta,w)$ belongs to $\mathcal H$.
    \item for every $f \in \mathcal H$ and $w \in U$ it holds that $f(w)=[f(\cdot),k(\cdot,w)]_{\mathcal{H}}$.
\end{enumerate}
\end{definition}
Equivalently,  reproducing kernel Hilbert spaces are characterized by the continuity of point evaluations. Now consider a  Herglotz-Nevanlinna function $h$ and  its maximal domain $\mathrm{dom}(h)$
(see Section \ref{SubSec:QH}). Such a function gives rise to a reproducing kernel Hilbert space, which we will call Herglotz space \cite[Chapter 1]{Branges1968}:
\begin{proposition}\label{PropHerglotzspace}
Let $h$ be a Herglotz-Nevanlinna function and set 
\begin{align*} 
N_h\vcentcolon \mathrm{dom}(h) \times \mathrm{dom}(h) \rightarrow \C, \quad 
N_h(\zeta,w)
:=\frac{h(\zeta)-\overline{h(w)}}{\zeta-\overline w}, \  \zeta \neq \overline w, \quad
N_h(w,\overline{w})=h'(w). 
\end{align*}
Then there exists a (unique) reproducing kernel Hilbert space $\mathcal{L}(h)$ with kernel $N_h$. 
\par\smallskip
Moreover, the multiplication operator by the independent variable
\begin{align*}
S_h \vcentcolon \{ f \in \mathcal{L}(h) \vcentcolon  \zeta \cdot f(\zeta) \in \mathcal{L}(h)\}\rightarrow \mathcal{L}(h), \quad S(f)(\zeta)= \zeta \cdot f(\zeta)
\end{align*}
\color{black}
is a simple symmetric operator  with deficiency index $(1,1)$. 
\par\smallskip
Finally, the linear relation 
\begin{align*}
A_h \vcentcolon = \{ (f,g)\in \mathcal{L}(h) \times \mathcal{L}(h) \vcentcolon \exists c \in \C \vcentcolon g(\zeta)-\zeta \cdot f(\zeta)\equiv c \}
\end{align*}
is a self-adjoint extension of $S_h$ which satisfies
\begin{align*}
    \mathrm{dom}(h)=\varrho(A_h) \quad \text{ and } \quad
    (A_h-w)^{-1}(f)(\zeta)=D_w(f)(\zeta)=\frac{f(\zeta)-f(w)}{\zeta-w}.
\end{align*} 
\end{proposition}
 It turns out that every simple symmetric operator with deficiency index $(1,1)$ is, up to an isometric isomorphism, of the above form (see e.g. \cite{PartFund}): 
\begin{proposition}\label{Prop:Symop1}
Let $S$ be simple symmetric operator with deficiency index $(1,1)$ on $\mathcal{H}$. Then there exists a Herglotz-Nevanlinna function $h$ and an isometric isomorphism 
\begin{align*}
    U \vcentcolon \mathcal{H} \rightarrow \mathcal{L}(h)
\end{align*}
such that $USU^*=S_h$.  
\end{proposition}
\subsection{The main theorem for simple symmetric operators}
Now we describe the extensions of simple symmetric operators in detail. Recall that for a given analytic function $f$ the difference quotient is denoted by  $D_w(f)$. Moreover, this operation satisfies the resolvent identity: 
\begin{align*}
    D_x(f)-D_w(f)=(x-w)\cdot D_x(D_w(f)).
\end{align*}
Let us define the following set of functions:
\begin{definition}
Let $h$ be a Herglotz-Nevanlinna function. Then we define $\mathcal{M}(h)$  as  
\begin{align*}
     \mathcal{M}(h) \vcentcolon &= \{ g \vcentcolon \mathrm{dom}(h) \rightarrow \C \text{ analytic } \ | \  D_w(g) \in \mathcal{L}(h) \quad  \forall w \in \mathrm{dom}(h) \} \\
     &=\{ g \vcentcolon \mathrm{dom}(h) \rightarrow \C \text{ analytic } \ | \  D_i(g) \in \mathcal{L}(h)   \}
\end{align*}
\end{definition}
Note that $\mathcal{L}(h) \subset \mathcal{M}(h)$, because $\mathcal{L}(h)$ is invariant under the difference quotient operator. Moreover, it is sufficient to demand that $D_i(g) \in \mathcal{L}(h)$, because in that case it holds that
\begin{align*}
    D_w(g)=\left(I+(w-i) \cdot D_w \right)D_i(g) \in  \mathcal{L}(h)   \quad \forall w \in \mathrm{dom}(h).
\end{align*}
These functions describe the regular extensions in the following way: 
\begin{corollary}\label{Cor:SimpleExt}
Let  $h$  be a Herglotz-Nevanlinna function, $S_h$ the multiplication operator by the independent variable on $\mathcal{L}(h)$ and $A_h$  the self-adjoint extension satisfying $(A_h-w)^{-1}=D_w$.
\par\smallskip
Then  every function  $g \in \mathcal{M}(h)\setminus\{0\}$ defines a regular extension of $S_h$ via
    \begin{align}\label{For1843905}
        (\tilde{A}-w)^{-1}(f) =D_w(f)-\frac{f(w)}{g(w)} \cdot D_w(g), \quad  w \in \varrho(A_h) \vcentcolon g(w)\neq 0.
    \end{align} 
And conversely, every regular extension is of this form. 
\end{corollary}
In other words, the regular extensions of $S_h$ correspond in a natural way to the analytic functions in $\mathcal{M}(h)$. 
\begin{proof}
    Let us fix the kernel function $\phi=N_h(\cdot,\overline{i})\in \eta_i(S)$ and the self adjoint extension $A_h$ as our elements of reference. First, we calculate the corresponding $\varphi$-field as follows:
    \begin{align*}
        \varphi_\phi(\zeta)&=(I+(\zeta-i)D_\zeta) N_h(\cdot,\overline{i}) =
        (I+(\zeta-i)D_\zeta) D_i(h)
        =D_\zeta(h)=N_h(\cdot,\overline{\zeta}).
    \end{align*}
    Here, we have used the identity $D_\zeta(h)=N_h(\cdot,\overline{\zeta})$ twice. 
    \par\smallskip
    Now let
    $v\in \mathcal{L}(h)$ be an arbitrary target vector, $q$ the associated normalized $Q$-function and $c\in \C$ a constant such that $q+c$ is not identically zero.
    Finally, let $w \in \varrho(A_h)$ such that $q(w)+c\neq 0$. Then the associated regular extension is determined by the following formula:  
    \begin{align*}
        (\tilde{A}-w)^{-1}(f)= (A_h-w)^{-1}(f)-\frac{[f , \varphi_\phi(\overline{w}) ]_\mathcal{H}}{q(w)+c} \cdot \varphi_v(w)= D_w(f)-\frac{f(w)}{q(w)+c} \cdot \varphi_v(w).
    \end{align*}
    We will show that this extension can be expressed as described in  equation \eqref{For1843905}.
    \par\smallskip
    To this end, we first show that $D_i(q+c)=v$ holds. Since $q+c$ is a $Q$-function associated to $v$, the following identity is true: 
    \begin{align*}
    D_i(q+c)&=[\varphi_v(\zeta), \varphi_\phi(\overline{i})]_{\mathcal{L}(h)}
    \end{align*}
    Next, we calculate the right hand side:
    \begin{align*}
    [\varphi_v(\zeta), \varphi_\phi(\overline{i})]_{\mathcal{L}(h)}
    &= [(I+(\zeta-i)D_\zeta) v, N_h(\cdot,i)]_{\mathcal{L}(h)}\\
    &=[ v, (I+(\overline{\zeta}-\overline{i})D_{\overline{\zeta}})N_h(\cdot,i)]_{\mathcal{L}(h)}
    =[ v, N_h(\cdot,\zeta)]_{\mathcal{L}(h)}
    =v(\zeta).
\end{align*}
Here we have used that $D_\zeta^*= D_{\overline{\zeta}}$.
In particular, it follows that  $ q+c \in \mathcal{M}(h)$ by construction. Finally, we calculate the $\varphi$-field induced by $v$ as follows:
\begin{align*}
        \varphi_v(\zeta)&=(I+(\zeta-i)D_\zeta) v =
        (I+(\zeta-i)D_\zeta) D_i(q+c)
        =D_\zeta(q+c).
    \end{align*}
This means that our above expression simplifies to the following identity: 
    \begin{align*}
        (\tilde{A}-w)^{-1}(f)= D_w(f)-\frac{f(w)}{q(w)+c} \cdot D_w(q+c)  \quad  w \in \varrho(A_h) \vcentcolon q(w)+c\neq 0.
    \end{align*}
    Consequently, we have shown that $(\tilde{A}-w)^{-1}$ is of the form described in equation $\eqref{For1843905}$ with $g=q+c$. 
    \par\medskip
    Conversely, let $g\in \mathcal{M}(h)$ and consider the following operator valued function:
    \begin{align*}
        T_{g} \vcentcolon \varrho(A) \setminus \{\zeta \vcentcolon g(\zeta) \neq 0\} \rightarrow \mathscr{L}(\mathcal{H}), \quad T_g(f)=D_w(f)-\frac{f(w)}{g(w)} \cdot D_w(g)
    \end{align*}
    Now let us \emph{define} $v\vcentcolon =D_i(g)\in \mathcal{L}(h)$. In the same spirit as before, we can show that 
    $\varphi_v(\zeta)=D_\zeta(g)$. This  implies that the following holds:
    \begin{align}\label{For:fdosil}
        [\varphi_v(\zeta), \varphi_\phi(\overline{w})]_{\mathcal{L}(h)}=
        [D_\zeta(g), N_h(\cdot,w)]_{\mathcal{L}(h)}=D_\zeta(g)(w).
    \end{align}
    This means that $g$ is a $Q$-function associated to $v$. Now let $q$ be the normalized $Q$-function associated to $v$. Then it follows that $g=q+c$ for some constant $c\in \C$. Now let $\tilde{A}$ be the regular extension associated to $v$ and $c$, which admits the following description of its resolvent: 
    \begin{align*}
        (\tilde{A}-w)^{-1}(f)&= (A_h-w)^{-1}(f)-\frac{[f , \varphi_\phi(\overline{w}) ]_\mathcal{H}}{q(w)+c} \cdot \varphi_v(w)= D_w(f)-\frac{f(w)}{q(w)+c} \cdot \varphi_v(w)\\
        &= D_w(f)-\frac{f(w)}{g(w)} \cdot D_w(g).
    \end{align*}
    In other words, we see that every function $g$ defines a regular extension via formula \eqref{For1843905}.
\end{proof}

It is important to note that for a given regular extension $\tilde{A}$ defined by the data $(v,c)$ the function $g$ coincides with $q+c$.
\color{black}
\subsection{The space $\mathcal{M}(h)$} \label{SectionMh} In this section, we establish an explicit representation of the space $\mathcal{M}(h)$.  For the sake of simplicity, we assume that the function $h$ has the following integral representation: 
    \begin{align*}
    h(\zeta)= a+ \int_{\R} \frac{1+t \zeta}{t-\zeta} \cdot \mathrm{d}\nu
\end{align*}
Note that our assumption is that $b$ in Proposition \ref{Prop:IntRep} equals zero. We start with describing the space $\mathcal{L}(h)$ in the following way:
\begin{proposition} \label{Prop:RealHnfunctions}
   Let us define $\mathrm{d} \mu = (1+t^2) \cdot \mathrm{d} \nu$.  Then the operator
    \begin{align*}
        F \vcentcolon \ L^2(\R,\mu) \rightarrow \mathcal{L}(h), \quad f  \mapsto \left(\zeta \mapsto 
        \int_{\R} \frac{f(t)}{t-\zeta} \cdot d\mu(t)\right)
    \end{align*}
    is an isometric isomorphism \cite[Theorem 5]{Branges1968}.
\end{proposition}
It should be noted that the elements in $\mathcal{L}(h)$ are quasi-Herglotz functions. Indeed, if $f$ is a positive function in $L^2(\nu)$, then $F(f)$ is a Herglotz-Nevanlinna function \cite[Theorem 3.9]{Teschl2009MathematicalMechanics}. Consequently, we can split up an arbitrary function $f\in L^2(\nu)$ into its positive and negative real imaginary parts, and end up with a sum of the form described in equation \eqref{QuasiHn}. 
\par\smallskip
Now let us consider the  Herglotz-Nevanlinna function $h$ given by $h_{|\C^+}=i$ and $h_{|\C^-}=-i$. It is well-known that its representing measure is just the Lebesque measure and that the associated space $\mathcal{L}(h)$ is of the following form:
\begin{align*}
        \mathcal{L}(h)=\mathcal{H}^2(\C^+) \oplus \mathcal{H}^2(\C^-).
    \end{align*}
    Here, $\mathcal{H}^2(\C^+)$ denotes the Hardy space, which we  embed into $\mathcal{L}(h)$ as follows: 
    \begin{align*}
        j \vcentcolon \mathcal{H}^2(\C^+)\hookrightarrow \mathcal{L}(h), \quad i(f)(\zeta) \vcentcolon = 
        \begin{cases}
            f(\zeta), \ &\text{if} \ \zeta \in \C^+ \\
        0, \ &\text{if} \ \zeta \in \C^-
        \end{cases}.
    \end{align*}
We see that there are sometimes non-trivial functions in $\mathcal{L}(h)\subset \mathcal{M}(h)$ that vanish identically on either $\C^+$ or $\C^-$. Of course, a necessary conditions for the existence of such functions is that $\sigma(A_h)=\R$ holds. More precisely, functions of this type can only exist if and only if the $L^2(\R,\mu)\neq H^2(\R,\mu)$, which was discussed in detail in \cite{Aleman}. It will turn out that these functions define regular extensions with interesting spectral properties. 
\par\smallskip
Finally, the space $\mathcal{M}(h)$ has then the following representation: 
\begin{proposition}\label{Propfunctionen}
Let us define $\mathrm{d} \mu = (1+t^2) \cdot \mathrm{d} \nu$.  Then
the space $\mathcal{M}(h)$ consists of the quasi-Herglotz functions 
\begin{equation*}
g(\zeta)=a +\int_\R  \frac{1+t \zeta }{t-\zeta} \cdot \left(\frac{f(t)}{t-i} \cdot \mathrm{d}\mu(t)\right)\quad \forall \zeta\in\mathbb \C\setminus \R. 
\end{equation*}
where $f \in L^2(\R,\mu)$ and $a\in \C$ is arbitrary. Note that $\frac{1}{t-i} \in L^2(\R,\mu)$ which implies that 
\begin{align*}
    \frac{f(t)}{t-i} \cdot \mathrm{d}\mu(t)
\end{align*}
is a (finite) complex measure by Hoelder`s inequality. 
\end{proposition}
\begin{proof}
By construction, a function $g$ is in $\mathcal{M}(h)$ if and only if $D_{\overline{i}}(g)$ is an element of $\mathcal{L}(h)$. This is the case if and only if there is a function $f\in  L^2(\R,\mu)$ such that 
\begin{align*}
D_{\overline{i}}(g)(\zeta)=\int_\R  \frac{1 }{t-\zeta} \cdot f(t) \cdot \mathrm{d}\mu(t)\quad \forall \zeta\in\mathbb \C\setminus \R,
\end{align*}
see Proposition \ref{Prop:RealHnfunctions}.
We  can rewrite this as follows:
\begin{align*}
g(\zeta)&=
g(\overline{i}) + (\zeta+i) \cdot \int_\R  \frac{1 }{t-\zeta} \cdot f(t) \cdot \mathrm{d}\mu(t)\\
&=g(\overline{i}) +  \int_\R  \left(\frac{\zeta+i }{t-\zeta} + \frac{i }{t-i}-\frac{i }{t-i}\right) \cdot f(t) \cdot \mathrm{d}\mu(t)\\
&=  \left(g(\overline{i}) +   \int_\R  \frac{i }{t-i} \cdot f(t) \cdot \mathrm{d}\mu(t)\right)+ \int_\R \left( \frac{-i }{t-i} + \frac{\zeta+i }{t-\zeta}\right) \cdot f(t) \cdot \mathrm{d}\mu(t)\\
&=a + \int_\R  \left(\frac{1+t \zeta }{t-\zeta} \cdot \frac{1}{t-i}\right) \cdot f(t) \cdot  \mathrm{d}\mu(t),
\end{align*}
where $a=g(\overline{i}) +   \int_\R  \frac{i }{t-i} \cdot f(t) \cdot \mathrm{d}\mu(t) $ is a constant.   Note that $g(\overline{i})$ can be chosen arbitrarily, because with $g$ also $g+c, \ c \in \C$ is an element of $\mathcal{M}(h)$.
\end{proof}

\subsection{A Functional model}
Under some mild additional assumption, we can model every extension as the difference-quotient operator on a suitable reproducing kernel space. To this end, we  recall the concept of a conjugation:
\begin{proposition}[Conjugation]\label{Prop:Conj}
Let $\mathcal{H}$ be a reproducing kernel Hilbert space with domain $U$ and kernel $k$. Moreover, let
$g\in \mathcal{O}(U)$ be an analytic function which is not identically zero on any component of $U$. 
Consider the vector space $g \cdot \mathcal{H} \subset \mathcal{O}(U)$ and the operator 
\begin{align*}
    M_g \vcentcolon \mathcal{H} \rightarrow g \cdot \mathcal{H}, \quad  f \mapsto g \cdot f,
\end{align*}
which is bijective by our assumption on $g$. We equip  $g \cdot \mathcal{H}$ with the inner product that turns  $M_g$ into an isomorphism of Hilbert spaces. Then $g \cdot \mathcal{H}$
is a reproducing kernel Hilbert space with kernel 
\begin{align*}
    g(\zeta) \cdot k(\zeta,w) \cdot \overline{g(w)}.
\end{align*}
\end{proposition}
A proof can be found in \cite{RepKernelSpaces}.
\begin{proposition}
Let $g\in \mathcal{M}(h)$ be such that $g$ is not identically zero on  neither $\C^+$ nor  $\C^-$. Moreover, consider the associated extension
\begin{align}\label{For:Leibniz}
    (\tilde{A}-w)^{-1}(f) =D_w(f)-\frac{f(w)}{g(w)} \cdot D_w(g).
\end{align}
Then the following diagram commutes:  
\[\begin{tikzcd}
\mathcal{L}(h) \arrow{r}{(\tilde{A}-w)^{-1}} \arrow[swap]{d}{\frac{1}{g}} & \mathcal{L}(h)  \\
\frac{1}{g} \mathcal{L}(h) \arrow{r}{D_w} & \frac{1}{g} \mathcal{L}(h) \arrow{u}{g}
\end{tikzcd}.
\]
In other words, if $ \mathcal{F} \vcentcolon \mathcal{L}(h) \rightarrow \frac{1}{g} \cdot \mathcal{L}(h)$ denotes the isomorphism induced by the conjugation, then it holds that
\begin{align*}
    \mathcal{F} (\tilde{A}-w)^{-1} \mathcal{F}^{-1} =D_w.
\end{align*}
\end{proposition}
\begin{proof}
Let $f$ and $g$ be analytic functions. 
First, we recall the Leibniz-rule for difference quotients: 
\begin{align*}
    D_w(g \cdot f)(\zeta) = g(\zeta)\cdot D_w(f)+ f(w) \cdot D_w(g).
\end{align*}
Now let $f \in \frac{1}{g} \mathcal{L}(h)$. Then we calculate 
\begin{align*}
    \mathcal{F} (\tilde{A}-w)^{-1} \mathcal{F}^{-1}(f)&=
    \frac{1}{g} \cdot (\tilde{A}-w)^{-1}(g \cdot f)
    = \frac{1}{g} \left( D_w(g \cdot f)-\frac{g(w) \cdot f(w)}{g(w)} \cdot D_w(g) \right)\\
    &=\frac{1}{g} \left(g \cdot D_w(f)+ f(w) \cdot D_w(g)-f(w) \cdot D_w(g) \right)=D_w(f).
\end{align*}
This completes the proof. 
\end{proof} 
\subsection{Example: Inner functions and Model spaces}
In this section, we explore a connection to another class of well studied reproducing kernel Hilbert spaces. We first give the following definition: 
\begin{definition}
    Let $v$ be a bounded analytic function on $\C^+$ such that $\|v\|_\infty\leq 1$. Then $v$ is an inner function, if it has unimodular boundary values almost everywhere. 
\end{definition}
In order to stay in the framework of this article, we extend $v$ to the $\C^-$ via the following rule: 
\begin{align*}
    v(\zeta)=\frac{1}{\overline{v(\overline{\zeta})}}, \quad \forall \zeta \in \C^-.
\end{align*}
This extension is usually called pseudocontinuation  \cite{Modelspaces}. Next, consider the Cayley transform 
\begin{align*}
h=i \cdot \frac{1- v}{1+v}.
\end{align*}
Then it is straightforward to verify that $h$ is a Herglotz-Nevanlinna function with real boundary values (a so called singular Herglotz-Nevanlinna function).  
\par\smallskip
Now let us consider the following function:
\begin{align*}
   g \vcentcolon = \frac{1}{2i}\cdot (h+i)=\frac{1}{2i}\cdot \left(\frac{1-v}{1+v}\cdot i +i \right) = \frac{1}{1+v}.
\end{align*}
Note that $g \in \mathcal{M}(h)$, because $D_i(g)=\frac{1}{2i} \cdot N_h(\cdot,\overline{i})$.
 The associated regular extension $\tilde{A}$  is then given by the following diagram:
\[\begin{tikzcd}
\mathcal{L}(h) \arrow{r}{(\tilde{A}-w)^{-1}} \arrow[swap]{d}{\frac{1}{g}} & \mathcal{L}(h)  \\
\frac{1}{g} \mathcal{L}(h) \arrow{r}{D_w} & \frac{1}{g} \mathcal{L}(h) \arrow{u}{g}
\end{tikzcd}.
\]
A short computation, whose details can be found in \cite{Me3}, shows that the kernel transforms in the following way: 
\begin{align*}
\frac{1+v(\zeta)}{\sqrt{2}}\cdot & N_h(\zeta,w) \cdot  \frac{1+\overline{v(w)}}{\sqrt{2}}= \frac{1-v(\zeta)\cdot \overline{v(w)}}{-i \cdot (\zeta-\overline w)}.
\end{align*}
Therefore, we see that $\frac{1}{g} \mathcal{L}(h)$ is the model space associated to $v$ (see \cite{Modelspaces} for a comprehensive introduction to their theory). Moreover, the   function $g=\frac{1}{v+1}$ is the defining function  of the extension $\tilde{A}$. We will see in the next section that the function theoretic properties of $g$ (and thus of $v$) are closely related to the spectral properties of $\tilde{A}$.

\section{The spectrum of the extensions}\label{Sec:hfsdljk}
Now let  $S$ be a symmetric relation with deficiency index $(1,1)$,  and let our reference framework be given by $A$ as the self-adjoint extension and $\phi \in \eta_i(S)$ as the defect element. 
\par\smallskip
In this section, we investigate the spectral properties of the regular extension $\tilde{A}$ given by the following resolvent formula:
    \begin{align}\label{ForHelp}
    (\tilde{A}-z)^{-1}= (A-z)^{-1}-\frac{[\ \cdot \ , \varphi_\phi(\overline{z}) ]_\mathcal{H}}{q(z)+c} \cdot \varphi_v(z) \quad  \forall z \in \varrho(A) \vcentcolon q(z)+c \neq 0.
    \end{align}
    Here, $v\in \mathcal{H}$ is a target vector, $q$ the associated normalized $Q$-function and $c \in \C$ is a constant such that $q+c$ is not identically zero. 
\begin{theorem}\label{Theorm:Sepctrum}
Let $w\in \varrho(A)$ be fixed. If $q(w)+c=0$ holds, then $w$ is a geometrically  simple eigenvalue of $\tilde{A}$. Moreover, if $w$ is a zero of order $n<\infty$, then  $w$ has algebraic multiplicity $n$.  
\end{theorem}
Here, it is important to remember that while the function $q+c$ is not allowed to vanish identically on $\varrho(A)$, it may vanish identically on either $\C^+$ or $\C^-$. We have discussed this possibility in detail in Section \ref{SectionMh}.  Of course, this can only happen when $\varrho(A)= \C \setminus \R$. 
\color{black}
\begin{proof}
    Let $z \in \varrho(A)$ such that $q(z)+c\neq 0$, which means that $z \in \varrho(\tilde{A})$. 
    First, we establish the following useful algebraic identity:
    \begin{align*}
        (\tilde{A}-w)f=0 &\Leftrightarrow (\tilde{A}-z) f +(z-w) f=0 \\
        &\Leftrightarrow f+ (z-w) (\tilde{A}-z)^{-1} f = 0 \\
    &\Leftrightarrow  f+ (z-w) (A-z)^{-1}(f)-\frac{(z-w) \cdot [f , \varphi_\phi(\overline{z}) ]_\mathcal{H}}{q(z)+c} \cdot \varphi_v(z)  = 0\\
     &\Leftrightarrow
     \left(I+(z-w)(A-z)^{-1} \right) 
    \left( f- \frac{(z-w) \cdot [f , \varphi_\phi(\overline{z}) ]_\mathcal{H}}{q(z)+c} \cdot \varphi_v(w)
     \right)=0.
    \end{align*}
    \color{black}
    Here, we have used the following identity in the last step:  
    \begin{align*}
        \left(I+(z-w)(A-z)^{-1} \right)\varphi_v(w)=\varphi_v(z).
    \end{align*}
    Consequently, in order for $(\tilde{A}-w)f=0$ to be true, $f$ has to be a multiple of $\varphi_v(w)$. This  implies that the geometric multiplicity of $w$ is at most one. 
\par\smallskip
Now set $f=\varphi_v(w)$.  We verify  using $q(w)=-c$ that $f$ is indeed an eigenvector:
\begin{align*}
        \varphi_v(w)-\frac{(z-w)[\varphi_v(w) , \varphi_\phi(w) ]_\mathcal{H}}{ q(z)+c} \cdot  \varphi_v(w)
        &=\varphi_v(w)-\frac{q(z)-q(w)}{ q(z)+c} \cdot  \varphi_v(w)\\
        &=\varphi_v(w)-\frac{q(z)+c}{ q(z)+c} \cdot  \varphi_v(w)
        =0.
\end{align*}
This means that $\lambda$ is a geometrically simple eigenvalue.
\par\smallskip
Now let us assume that $\lambda$ is a zero of order $n<\infty$. Then the resolvent $(\tilde{A}-\zeta)^{-1}$ has a pole of order $n$ at $\lambda$. The order of the pole coincides with the ascent at $\lambda$ \cite[Theorem 10.1]{FuncAnaOld}, i.e. the
 the smallest number $m$ such that   
\begin{align*}
    \mathrm{ker}((\tilde{A}-\lambda)^{m}) = \mathrm{ker}((\tilde{A}-\lambda)^{m+1}).
\end{align*}
Since $\lambda$ is geometrically simple this ascent coincides with the algebraic multiplicity of $\lambda$. This completes the proof.
\end{proof}
The distribution of zeros of $q+c$ is well behaved. More precisely, the function $q_{|\C^+}+c$ is of bounded type, since Herglotz-Nevanlinna functions restricted to $\C^+$ are of bounded type \cite[Chapter 5]{Rosenblum1994TopicsFunctions}. This means that  $q_{|\C^+}+c$ is either identically zero or has countably many zeros which  satisfy the Blaschke condition 
\begin{align*}
            \sum_{k=1}^\infty \bigg|\mathrm{Im}\left(\frac{1}{z_k}\right)\bigg|<\infty. 
\end{align*}
Here, $(z_k)_{k \in \N}\subset \C^+$ denotes the zeros counted with their multiplicities. This allows us to characterize the spectrum in the upper-half plane or lower-half plane: 
\begin{corollary}
    One of the following statements is true: 
    \begin{itemize}
        \item Every point $w\in\C^+$ is an eigenvalue of $\tilde{A}$.
        \item There exists a countable set of eigenvalues $(z_k)_{k \in \N}\subset \C^+$ (counted with regards to their algebraic multiplicity),  which satisfies the Blaschke condition 
        \begin{align*}
            \sum_{k=1}^\infty \bigg|\mathrm{Im}\left(\frac{1}{z_k}\right)\bigg|<\infty. 
        \end{align*}
    \end{itemize}
     Conversely, any such distribution of the spectrum occurs. 
    An analogous statement is  true for $\C^-$. 
\end{corollary}
\begin{proof}
    It is clear that the spectrum of an extension is of this form from the preceeding discussion. Conversely, let us show that any such distribution occurs. To this end, consider the Herglotz-Nevanlinna function $h$ defined as follows: 
    \begin{align*}
        h(\zeta)=\
        \begin{cases}
            i, \quad & \text{ if }\zeta \in \C^+ \\
            -i, \quad & \text{ if } \zeta \in \C^-.
        \end{cases}
    \end{align*}
    Now consider any distribution of eigenvalues on $\C^+$ as described above. Then we can find some function $g_+ \in \mathcal{H}^2(\C^+)$  whose zeros match this distribution \cite[Chapter 5]{Rosenblum1994TopicsFunctions}. Finally, let $g_- \in \mathcal{H}^2(\C^-) \setminus \{0\}$ be \emph{non- trivial}. 
    Then the function $g$ defined via $g_{|\C^+}=g_+$ and $g_{|\C^-}=g_-$ is an element of $ \mathcal{M}(h)  \setminus \{0\}$ (see Section \ref{SectionMh}). This function $g$ defines the following regular extension of the simple symmetric operator $S_h$ (see Chapter \ref{Ext:SimpleSym})
\begin{align*}
        (\tilde{A}-w)^{-1}(f) =D_w(f)-\frac{f(w)}{g(w)} \cdot D_w(g).
    \end{align*} 
Since the zeros of $g$ on $\C^+$ coincide with the prescribed eigenvalue distribution, we conclude that $\tilde{A}$ meets our demand. 
\color{black}
\end{proof}
So far we have only investigated points in $\varrho(A)$. In the final part of this section, we turn our attention towards points in the spectrum of $A$. To this end, we assume from now on that $S$ is a simple symmetric operator with deficiency index $(1,1)$. In the spirit of Chapter \ref{Ext:SimpleSym}, we can regard $S$ as the multiplication operator on a Herglotz space $\mathcal{L}(h)$ and $A=A_h$ as the self adjoint extension generating the difference quotient operator. Then  any regular extension $\tilde{A}$ is given by the following formula: 
\begin{align}\label{For:ldjsk}
        (\tilde{A}-w)^{-1}(f) =D_w(f)-\frac{f(w)}{g(w)} \cdot D_w(g), \quad  w \in \varrho(A_h) \vcentcolon g(w)\neq 0.
    \end{align} 
Recall that in this language the function $g$ plays the same role as $q+c$. 
    Now let us consider an isolated eigenvalue $w$ of the relation $A_h$, which has to be of order one \cite{Beherndt}. This means that $h$ has a pole at $w$ of order one, and every function in $\mathcal{M}(h)$ is either analytic at $w$ or has a pole of order one there (see Proposition \ref{Propfunctionen}). 
    This allows us to state the following theorem: 
\begin{theorem}
    Let $w \in \R$ be an isolated eigenvalue of $A_h$. Then the following holds: 
    \begin{itemize}
        \item If  $g(w)= \infty$, then $w$ is an element of $\varrho(\tilde{A})$.
        \item If $0 \neq g(w)\neq \infty$, then  $w \in \R$ is an isolated eigenvalue of $\tilde{A}$ of algebraic multiplicity one.
        \item If $g(w)=0$ and $w$ is a zero of order $n$, then $w$ is an eigenvalue of algebraic multiplicity $n+1$.
    \end{itemize}
\end{theorem}

The following proof also sheds some light on how to switch between different self-adjoint extensions as the chosen reference relation. 
\begin{proof}
   First, let us consider the self-adjoint extension $A_{-\frac{1}{h}}$, which generates the difference quotient operator on  $\mathcal{L}\left( -\frac{1}{h}\right)$. This  self-adjoint extension fits into the following commuting diagram \cite{Beherndt}: 
    \[\begin{tikzcd}
\mathcal{L}(h) \arrow{r}{(A_{-\frac{1}{h}}-w)^{-1}} \arrow[swap]{d}{\frac{1}{h}} & \mathcal{L}(h)  \\
 \mathcal{L}\left( -\frac{1}{h}\right) \arrow{r}{D_w} & \mathcal{L}\left( -\frac{1}{h}\right) \arrow{u}{h}
\end{tikzcd}.
\]
Then it is easy to check that we can represent the regular extension $\tilde{A}$ by the following two commuting diagrams: 
   \[\begin{tikzcd}
\mathcal{L}(h) \arrow{r}{(\tilde{A}-w)^{-1}} \arrow[swap]{d}{\frac{1}{g}} & \mathcal{L}(h)  \\
\frac{1}{g} \mathcal{L}(h) \arrow{r}{D_w} & \frac{1}{g} \mathcal{L}(h) \arrow{u}{g}
\end{tikzcd} \quad \quad 
\begin{tikzcd}
\mathcal{L}\left( -\frac{1}{h}\right) \arrow{r}{(\tilde{A}-w)^{-1}} \arrow[swap]{d}{\frac{h}{g}} & \mathcal{L}\left( -\frac{1}{h}\right)  \\
\frac{1}{g} \mathcal{L}\left( -\frac{1}{h}\right) \arrow{r}{D_w} & \frac{1}{g} \cdot \mathcal{L}\left( -\frac{1}{h}\right) \arrow{u}{\frac{g}{h}}
\end{tikzcd}.
\]
In other words, switching to the self-adjoint extension $\mathcal{A}_{-\frac{1}{h}}$ leads to a multiplication of the defining function $g$ by $\frac{1}{h}$. Finally, it is well known that if  $A_h$ has an isolated eigenvalue at $w$, then $w \in \varrho(A_{-\frac{1}{h}})$ \cite{Beherndt}. Consequently, we can use Theorem \ref{Theorm:Sepctrum} with $A_{-\frac{1}{h}}$ as the base relation and $\frac{g}{h}$ as the defining function. Recall that   $h$ has a pole of order one at $w$. Thus if $g$ has a pole (of order one) there as well, then  $\frac{g}{h}(w)\neq 0$. This then means that $w \in \varrho(\tilde{A})$. The other cases are treated accordingly.  
\end{proof}
Finally, we turn out attention towards the essential part of the spectrum. For an arbitrary operator $B$ it is defined as follows: 
\begin{align*}
    \sigma_{ess}(B)&\vcentcolon 
    =\sigma(B) \setminus   \{x \text{ is an isolated eigenvalue of finite algebraic multiplicity}\}. 
\end{align*}
We point out that there are many different definitions of essential spectra for non-self adjoint operators \cite{EssentialSPectrum}. Moreover, we recall that the  essential spectrum of all self-adjoint extensions of $S$ coincides \cite[Section 3]{EssSpectra}. However, it is not an invariant of arbitrary extensions. Indeed, we just proved that a suitable simple symmetric operator $S$  has an extension $\tilde{A}$ such that  $\C^+\subset \sigma(\tilde{A})$, which means that $\sigma_{ess}(\tilde{A})=\overline{\C^+}$. It turns out, that this edge case is the only obstruction for the essential spectrum to be an invariant:
\begin{theorem}
    Let  $S_h$ be the multiplication operator by the independent variable on some Herglotz space $\mathcal{L}(h)$, $A_h$ the self-adjoint extension giving rise to the difference quotient operator, and $\tilde{A}$  a regular extension defined by some function $g \in \mathcal{M}(h)$ (see formula \eqref{For:ldjsk}).
    Then the following holds:
\color{black}
    \begin{align*}
        \sigma_{ess}(A_h)\cap \R = \sigma_{ess}(\tilde{A}) \cap \R.
    \end{align*}
    Moreover, if the function $g$  is not identically zero on neither $\C^+$ nor $\C^-$, then it even holds that 
    \begin{align*}
        \sigma_{ess}(A_h) = \sigma_{ess}(\tilde{A}).
    \end{align*}
\end{theorem}
\begin{proof}
If the function  $g$ is identically zero on $\C^+$ (or $\C^-$), then $\sigma_{ess}(\tilde{A}) = \overline{\C^+}$.  In this case,  $\sigma_{ess}(A_h)$ has to coincide with the real line (see Section \ref{SectionMh}), which means that we established our claim.
\par\smallskip
Therefore, we can assume without loss of generality that 
$g$ is not identically zero on neither $\C^+$ nor $\C^-$. Then it holds that $\sigma_{ess}(\tilde{A}) \subset \R$, because  $g$ can only have zeros of finite order on $\C\setminus \R$.  Now, we show that $\sigma_{ess}(A_h) = \sigma_{ess}(\tilde{A})$. 
\par\smallskip
First, let us assume that $w \notin \sigma_{ess}(A_h)$. This means that $w$ is at worst an isolated eigenvalue, and we can find an $\epsilon>0$ such that $(w-\epsilon,w+\epsilon)\cap \sigma(A_h)$ consists of at most $w$. Then the function $g \vcentcolon \varrho(A) \rightarrow \C$ is meromorphic on the  set $(w-\epsilon,w+\epsilon)$. This means that $(w-\epsilon,w+\epsilon) \cap \sigma(\tilde{A}_h)$ consists only of the zeros 
 $g$  and possibly $w$. All these points are eigenvalues of finite multiplicity and they do not accumulate towards $w$ (because $g$ is not identically zero).  Therefore, $w$ is not an element of $ \sigma_{ess}(\tilde{A})$.
\par\smallskip
Conversely, let $w\notin \sigma_{ess}(\tilde{A})$, which means that  $w$ is at worst an eigenvalue of finite order. Moreover, let us assume, for sake of contradiction, that $w\in \sigma_{ess}(A_h)$. We define the following two functions in $\mathcal{L}(h)$ (see Proposition \ref{Prop:RealHnfunctions}): 
\begin{align*}
     f_1(\zeta) \vcentcolon = \int_{\R \setminus (w-\alpha,w+\alpha)} \frac{1}{t-\zeta} \cdot d\mu(t) \quad \quad   
      f_2(\zeta) \vcentcolon = \int_{ (w-\alpha,w+\alpha)} \frac{1}{t-\zeta} \cdot d\mu(t) 
\end{align*}
Here $\alpha>0$ is chosen small enough such that both $f_1$ and $f_2$ are non-trivial and that the set 
\begin{align*}
    \sigma(\tilde{A})\cap (w-\alpha,w+\alpha)
\end{align*}
consists only (possibly) of $w$. We note that $f_1$ is analytic on $(w-\alpha,w+\alpha)$ and that $f_2$ has no meromorphic continuation on $(w-\alpha,w+\alpha)$. The latter is true because  $w \in \sigma_{ess}(A_h)$ \cite{Beherndt}. 
\par\smallskip
Now let $z \in \C^+$ be fixed. Then  the function
\begin{align}\label{for:hdsasa}
    x \mapsto [(\tilde{A}-x)^{-1}(f),N_h(\cdot,z)]= D_x(f)(z)-\frac{f(x)}{g(x)} \cdot D_x(g)(z)
\end{align}
has to have a meromorphic continuation through $(w-\alpha,w+\alpha)$ for any $f \in \mathcal{L}(h)$, because $\tilde{A}$ has at most one  isolated eigenvalue of finite multiplicity in this interval.
Let us first plug  $f_1$ into formula \eqref{for:hdsasa}. Then we infer that $g$ has to be meromorphic. In the second step, let us plug  $f_2$ into \eqref{for:hdsasa}. Then, using the fact that $g$ is meromorphic there, it is readily verified that the resulting function cannot be meromorphic on $(w-\alpha,w+\alpha)$. Therefore, we end up with a contradiction, which concludes our proof. 
\color{black}
\end{proof}
\section{Examples}
We decided to postpone an in-depth discussion of various examples to the follow-up paper, in which we will interpret regular extensions as asymmetric singular perturbations. For now, we only briefly discuss three  common sources of asymmetry. 
\par\smallskip
First, let $S_1$ and $S_2$ be two (densely defined) symmetric operators  with deficiency index $(1,1)$ on Hilbert spaces $\mathcal{H}_1$ and $\mathcal{H}_2$. Then  we can consider the self-adjoint extensions of $S_1 \oplus S_2$ in the Krein space $\mathcal{K} = \mathcal{H}_1 \oplus (-\mathcal{H}_2)$. Such an  extension $\tilde{A}$ is called  partially fundamentally reducible  \cite{PartFund}, and it can be described by a Krein type-resolvent formula  of the following form \cite[Theorem 4.7]{PartFund}: 
\begin{align*}
    (\tilde{A}-w)^{-1} = 
 (A-w)^{-1}-\frac{[\ \cdot \ ,  \gamma(\overline{w}) ]_{\mathcal{K}}}{\tilde{g}(w)} \cdot  \gamma(w).
\end{align*}
Here,  $A$ is an operator which decomposes into an orthogonal sum $A=A_1\oplus A_2$ of self-adjoint operators on $\mathcal{H}_1$ and  $\mathcal{H}_2$, and is therefore self-adjoint on the Hilbert space  $\mathcal{H}_1 \oplus \mathcal{H}_2$. Finally, $\gamma \vcentcolon  \varrho(A) \rightarrow \mathcal \mathcal{K}$ and $\tilde{g} \vcentcolon \varrho(A) \rightarrow \mathcal \C$ are suitable generalizations of the Weyl-solution and Weyl-function. Either way, this construction is  a special case of our extension theory. Indeed,  let us consider the following bounded functional on $\mathrm{dom}(A)$:
\begin{align*}
    E \vcentcolon \mathrm{dom}(A) \rightarrow \C, \quad 
    E(y) \vcentcolon =  [(A-w) y ,  \gamma(\overline{w}) ]_{\mathcal{K}}
\end{align*}
In this case, it is well-known  that the operator 
\begin{align*}
    S \vcentcolon = A_{|\{u \in \mathrm{dom}(A) \vcentcolon E(u)=0\}} 
\end{align*}
 is a symmetric operator with deficiency index $(1,1)$ on the Hilbert space $\mathcal{H}_1 \oplus \mathcal{H}_2$ \cite{KurasovPert}. Furthermore, it is straightforward to check that $A$ and $\tilde{A}$ are regular extensions of $S$.  This means that the Krein-type resolvent formula above is a special case of the theory  developed in this article.
 Finally, we point out that this class includes many interesting examples including certain indefinite Sturm-Lioville operators \cite{OpenProblem} and operator models for meromorphic functions of bounded type \cite{Me3}. 
 \color{black}
 \par\medskip
 Next, we discuss certain differential operators with asymmetric boundary conditions. For example, let us consider the space $\mathcal{H}=L^2([0,2])$ and the operator $D(f)=-i \cdot f^\prime$. Then it is well-known that a symmetric operator with deficiency index $(1,1)$ is given by the  boundary condition
\begin{align*}
    f(0)=f(2 \pi )=0.
\end{align*}
Moreover, let us consider the following two boundary conditions: 
\begin{align*}
    f(0)=f(2 \pi ), \quad \text{ and } \quad f(0)=0.
\end{align*}
The first boundary condition leads  to a self-adjoint extension and the second one to a regular extension. The latter  extension is of particular interest, because it has empty spectrum \cite{SingPertubation2}. 
\par\medskip
Finally, we have discussed regular extensions of simple symmetric operators in depth in this article. In doing so, we  relied on the language of reproducing kernels, which turns out to be the most convenient setting for this type of analysis. 
We conclude this article with a complementary  discussion using the Sturm-Liouville model. 
\par\smallskip
To this end, we consider the minimal operator associated to the Sturm-Liouville differential expression $\tau = \frac{\mathrm{d}^2}{\mathrm{d}^2 x }$ on the space $L^2((0,\infty))$.  The maximal and minimal operators are given by the following respective domains: 
\begin{align*}
    \mathcal{D}_{max} &= \{ f \in  L^2((0,\infty)) \vcentcolon f \text{ and } f^\prime \text{ are absolutely continuous on } [0,\infty) \text{ and } f^{\prime\prime} \in L^2((0,\infty))\} \\
    \mathcal{D}_{min} &= \{ f \in  \mathcal{D}_{max} \vcentcolon f(0)=0, f^\prime(0)=0\} 
\end{align*}
The minimal operator is a simple symmetric operator with deficiency index $(1,1)$. A self-adjoint extension $A_0$ of the minimal operator is given by the restriction of the maximal operator to the following domain: 
\begin{align*}
    \mathcal{D}_{A_0}=  \{ f \in  \mathcal{D}_{max} \vcentcolon f(0)=0\} 
\end{align*}
We can define the following bounded functional on $\mathrm{dom}(A)$:
\begin{align*}
 E \vcentcolon \mathrm{dom}(A_0) \rightarrow \C, \quad  E(f) \vcentcolon =   f^\prime(0).
\end{align*}
Now let $g \in L^2((0,\infty))$ and consider the following operator: 
\begin{align*}
    \tilde{A} \vcentcolon \mathcal{D}_{A_0} \rightarrow L^2((0,\infty)), \quad \tilde{A}(f)\vcentcolon=A_0(f) + E(f) \cdot g =f^{\prime\prime}+f^\prime(0) \cdot g
\end{align*}
It can be shown that $\tilde{A}$ is a regular extension. Thus we end up with an extra term with a frozen argument, which were discussed on multiple occasions in the past \cite{FrozenArgument}. We note that not all  regular extensions of the minimal operator are of this form. A complete description of them will be given in the subsequent article dealing with  asymmetric singular perturbations.

\printbibliography

\end{document}